\documentclass[12pt, reqno]{amsart}
\usepackage{amsmath, amsthm, amscd, amsfonts, amssymb, graphicx, color}
\usepackage[bookmarksnumbered, colorlinks, plainpages]{hyperref}
\hypersetup{colorlinks=true,linkcolor=red, anchorcolor=green, citecolor=cyan, urlcolor=red, filecolor=magenta, pdftoolbar=true}

\newtheorem{theorem}{Theorem}[section]
\newtheorem{lemma}[theorem]{Lemma}

\newtheorem{corollary}[theorem]{Corollary}
\theoremstyle{definition}
\newtheorem{definition}[theorem]{Definition}

\theoremstyle{remark}
\newtheorem{remark}[theorem]{Remark}
\numberwithin{equation}{section}

\DeclareMathOperator{\diver}{div}
\DeclareMathOperator{\supp}{supp}
\DeclareMathOperator{\loc}{loc}

\def\RR{\mathbb{R}}

\newcommand{\R}{\mathbb{R}}
\newcommand{\HH}{\mathbb{H}}

\begin{document}

\setcounter{page}{1}

\title[Higher-order evolution inequalities on the  Kor\'{a}nyi  ball]{Higher-order evolution inequalities with Hardy potential on the  Kor\'{a}nyi ball}

\author[M. Jleli, M. Ruzhansky, B. Samet, B. T. Torebek]{Mohamed Jleli, Michael Ruzhansky, Bessem Samet, Berikbol T. Torebek$^*$}

\address{\textcolor[rgb]{0.00,0.00,0.84}{Mohamed Jleli \newline Department of Mathematics, College of Science, King Saud University, \newline Riyadh 11451, Saudi Arabia}}
\email{\textcolor[rgb]{0.00,0.00,0.84}{jleli@ksu.edu.sa}}
\address{\textcolor[rgb]{0.00,0.00,0.84}{Michael Ruzhansky  \newline Department of Mathematics: Analysis, Logic and Discrete Mathematics \newline Ghent University, Belgium
\newline and
\newline School of Mathematical Sciences \newline Queen Mary University of London \newline United Kingdom}}
\email{\textcolor[rgb]{0.00,0.00,0.84}{michael.ruzhansky@ugent.be}}
\address{\textcolor[rgb]{0.00,0.00,0.84}{Bessem Samet \newline Department of Mathematics, College of Science, King Saud University, \newline Riyadh 11451, Saudi Arabia}}
\email{\textcolor[rgb]{0.00,0.00,0.84}{bsamet@ksu.edu.sa}}
\address{\textcolor[rgb]{0.00,0.00,0.84}{Berikbol T. Torebek \newline Department of Mathematics: Analysis,
Logic and Discrete Mathematics \newline Ghent University, Krijgslaan 281, Ghent, Belgium \newline and \newline Institute of
Mathematics and Mathematical Modeling \newline 125 Pushkin str.,
050010 Almaty, Kazakhstan}}
\email{\textcolor[rgb]{0.00,0.00,0.84}{berikbol.torebek@ugent.be}}


\let\thefootnote\relax\footnote{$^{*}$Corresponding author}



\begin{abstract} We consider a higher order in (time) semilinear evolution inequality posed on the Kor\'{a}nyi ball under an inhomogeneous Dirichlet-type boundary condition. The problem involves an
inverse-square  potential $\lambda/|\xi|_\HH^2$, where $\lambda \geq -(Q-2)^2/4$ and a general weight function $V$ depending on the space variable  in front of the power nonlinearity. We first establish a general nonexistence result for the considered problem. Next, in the special case $V(\xi):=|\xi|_\HH^a$, $a\in \RR$, we prove the sharpness of our nonexistence result and show that the problem admits three different critical behaviors according to the value of the parameter $\lambda$.
\end{abstract}
\maketitle

\tableofcontents

\section{Introduction}
The purpose of this paper is to study the  existence and nonexistence of weak solutions to evolution inequalities of the form
\begin{equation}\label{P}
\frac{\partial^ku}{\partial t^k}-\frac{1}{\psi}\Delta_\HH u+\frac{\lambda}{|\xi|_\HH^2}u\geq V |u|^p\quad \mbox{in } (0,\infty)\times B_\HH, 	
\end{equation}
where $u=u(t,\xi)$, $k\geq1 $ is an integer, $B_\HH$ is the unit Kor\'{a}nyi ball of the Heisenberg group $\HH^N=(\R^{2N+1},\circ)$,  $\lambda \geq -\left(\frac{Q-2}{2}\right)^2$, $Q=2N+2$ is the homogeneous dimension of $\HH^N$, $p>1$,  $V=V(\xi)$ is a measurable function and $V>0$ a.e. in $B_\HH$. The term $\psi$ and $\Delta_\HH$ (the Kohn-Laplacian operator) will be defined in the next section. Notice that $\left(\frac{Q-2}{2}\right)^2$ is the best constant in the Heisenberg version of the Hardy inequality (see e.g. \cite{GA}, and also \cite{RUS})
$$\int\limits_{\mathbb{H}^N}|\nabla_{\mathbb{H}}v(\xi)|^2d\xi\geq \left(\frac{Q-2}{2}\right)^2\int\limits_{\mathbb{H}^N}\frac{|v(\xi)|^2}{|\xi|^2_{\mathbb{H}}}d\xi,\,\,v\in C^\infty_c\left(\mathbb{H}^N\right),$$
where $\nabla_{\mathbb{H}}$ is a horizontal gradient on $\mathbb{H}^N$ (will be defined in the next section).

We shall study \eqref{P} under the inhomogeneous Dirichlet-type boundary condition
\begin{equation}\label{BC}
u\geq f \quad \mbox{on }(0,\infty)\times \partial B_\HH, 	
\end{equation}
where $f=f(\xi)\in L^1(\partial B_\HH)$.

The study of existence and nonexistence for elliptic problems involving a Hardy potential in bounded and unbounded domains of $\RR^N$ has been considered in several works. For instance, Brezis-Dupaigne-Tesei \cite{BR}
investigated  the existence and nonexistence of solutions $u\geq 0$ to the elliptic equation
\begin{equation}\label{P-Brezis}
-\Delta u -\frac{\lambda}{|x|^2}u= u^p
\end{equation}
in a ball $B(0,R)$ of $\RR^N$, $N\geq 3$, where $p>1$ and  $0<\lambda\leq \left(\frac{N-2}{2}\right)^2$.  They proved the following results. Let $\sigma:=\frac{N-2}{2}-\sqrt{\left(\frac{N-2}{2}\right)^2-\lambda}$.
\begin{itemize}
\item[(i)] 	If $1<p<1+\displaystyle\frac{2}{\sigma}$, then there exists a nontrivial solution to \eqref{P-Brezis} in $\mathcal{D}'(B(0,R))$ with $u^p$ and $\displaystyle\frac{u}{|x|^2}\in L^1(B(0,R))$;
\item[(ii)] If $p\geq 1+\displaystyle\frac{2}{\sigma}$ and $u\in L^p_{\loc}(B(0,R)\backslash\{0\})$, $u\geq 0$, is a supersolution to \eqref{P-Brezis} in $\mathcal{D}'(B(0,R)\backslash\{0\})$, then $u\equiv 0$.
\end{itemize}
Problems of type \eqref{P-Brezis} have been also studied in exterior domains  \cite{Chen,CI,MO,WA} and cone-like domains \cite{FA,GH,LA,LI}. Further references in this direction are \cite{AB,Chen2,Chen3,Chen4,DU,FA2}.

Evolution problems involving a Hardy potential in bounded and unbounded domains of $\RR^N$ have been also studied by several authors. For instance, Baras and Golshtein \cite{BG84} discovered a critical behavior of the linear Cauchy problem
\begin{equation}\label{BG1}
\left\{
\begin{array}{l}
u_t=\Delta u-\frac{\lambda}{|x|^2}u, \,t>0,\,x\in\mathbb{R}^N,\\{}\\
u(0,x)=u_0(x),\,x\in \mathbb{R}^N.
\end{array}
\right.
\end{equation}
They proved that:
\begin{itemize}
\item[(i)] if $\lambda<-\left(\frac{N-2}{2}\right)^2$ then the problem \eqref{BG1} has no nonnegative nontrivial solutions;
\item[(ii)] if $\lambda\geq-\left(\frac{N-2}{2}\right)^2,$ then there exist a positive weak solution of the problem \eqref{BG1}.
\end{itemize}
Some extensions and developments of above results can be found in the papers \cite{Brezis1, Cabre, Gold1, Gold2, Gold3, Iagar, Kombe, Zhang, Vaz}.

In \cite{HA} Hamidi and Laptev considered the higher-order evolution inequality
\begin{equation}\label{P-HL}
\frac{\partial^ku}{\partial t^k}-\Delta u+\frac{\lambda}{|x|^2}u\geq |u|^p\quad \mbox{in }(0,\infty)\times \mathbb{R}^N
\end{equation}
subject to the initial condition
\begin{equation}\label{IC-HA}
\frac{\partial^{k-1}u}{\partial t^{k-1}}(0,x)\geq 0\quad \mbox{in }  \mathbb{R}^N,
\end{equation}
where $N\geq 3$, $\lambda \geq -\left(\frac{N-2}{2}\right)^2$ and $p>1$. They proved that,   if one of the following assumptions  is satisfied:
$$
\lambda\geq 0,\,\, 1<p\leq 1+\frac{2}{\frac{2}{k}+s^*};
$$
or
$$
-\left(\frac{N-2}{2}\right)^2\leq \lambda<0,\,\, 1<p\leq 1+\frac{2}{\frac{2}{k}-s_*},
$$
where
$$
s^*=\frac{N-2}{2}+\sqrt{\lambda+\left(\frac{N-2}{2}\right)^2}, \,  s_*=s^*+2-N,
$$
then \eqref{P-HL}--\eqref{IC-HA} admits no nontrivial weak solution. Recently, an improvement of the above result has been obtained in \cite{Jleli-Samet}. Namely, it was proved that, if $k\geq 2$ and
$$
\frac{2}{k}\left(N-2+\frac{2}{k}\right)\leq \lambda<2N,
$$
then for all $p>1$, \eqref{P-HL}--\eqref{IC-HA} admits no nontrivial weak solution. We also refer to  \cite{JSV} , where \eqref{P-HL} has been studied in an exterior domain of $\RR^N$ under various types of inhomogeneous boundary conditions. Other related works can be found in   \cite{ABD2,ABD3,ABD4}.

Elliptic and evolution problems in the Heisenberg group have been widely investigated, see e.g.
\cite{BI,BOD,BO,Ruz1,Ruz2,DA,FI,GA2,JKS,Kirane,PO,RU,Ruz3,ZI} and the references therein. In particular,  Zixia and  Pengcheng \cite{ZI} extended the obtained results in \cite{HA} to the Heisenberg group. Namely, they considered the problem
\begin{equation}\label{P-Z}
\frac{\partial^ku}{\partial t^k}-\Delta_\HH u+\frac{\lambda \psi}{|\xi|_\HH^2}u\geq |\xi|_\HH^\sigma |u|^p\quad \mbox{in } (0,\infty)\times \HH^N
\end{equation}
subject to the initial condition
\begin{equation}\label{IC-Z}
\frac{\partial^{k-1}u}{\partial t^{k-1}}(0,\xi)\geq 0\quad \mbox{in }  \HH^N,
\end{equation}
where $p>1$, $-2< \sigma\leq 0$ and $\lambda \geq -\left(\frac{Q-2}{2}\right)^2$. They proved that,   if one of the following assumptions  is satisfied:
$$
\lambda\geq 0,\,\, 1<p\leq 1+\frac{2+\sigma}{\frac{2}{k}+s^*};
$$
or
$$
-\left(\frac{Q-2}{2}\right)^2\leq \lambda<0,\,\, 1<p\leq 1+\frac{2+\sigma}{\frac{2}{k}-s_*},
$$
where
$$
s^*=\frac{Q-2}{2}+\sqrt{\lambda+\left(\frac{Q-2}{2}\right)^2}, \,  s_*=s^*+2-Q,
$$
then \eqref{P-Z}--\eqref{IC-Z} admits no nontrivial weak solution.

To the best of our knowledge, the study of existence and nonexistence for evolution inequalities with an inverse-square potential  on bounded domains of the Heisenberg group has not been previously considered in the literature. Motivated by this fact, problem \eqref{P} under the boundary condition \eqref{BC} is investigated in this paper.

The rest of the paper is organized as follows.  In Section \ref{sec2}, we recall some notions and properties related to the Heisenberg group. In Section \ref{sec3}, after defining weak solutions to \eqref{P}--\eqref{BC}, we state our main results and discuss some special cases of the weight function $V$. Section \ref{sec4} is devoted to some auxiliary results. Finally, the proofs of our main results are provided in Section \ref{sec5}.

Let us say some words about  our approach. The proofs of the nonexistence results make use of the nonlinear capacity method (see e.g. \cite{MI})  that requires the construction  of appropriate test functions. This construction is essentially related to the considered domain, the differential operator  $\frac{1}{\psi}\Delta_\HH+\frac{\lambda}{|\xi|_\HH^2}$ and the boundary condition \eqref{BC}. Our existence results are established by the construction of explicit solutions.

Throughout this paper, the letter $C$ denotes a generic positive constant independent of the scaling parameters $T,R$ and the solution $u$.   The value of $C$ is not necessarily the same  from one line to         another. The notation  $s\gg 1$, where $s>0$,   means that  $s$ is sufficiently large.

\section{Preliminaries}\label{sec2}

For the reader's convenience, we recall  in this section some notions and properties related to the Heisenberg group. For more details, we refer to \cite{AB,BI,DA,FI,FO,GA} and the references therein.

The Heisenberg group group $\HH^N$, $N\geq 1$,  is the set $\R^{2N+1}$ equipped  with the group law $\circ$ given by
$$
\xi\circ \xi':=\left(x+x',y+y',\phi+\phi'+2 \sum_{i=1}^N\left(x_i'y_i-x_iy_i'\right)\right),\quad \xi,\xi'\in \R^{2N+1},
$$
where $$\xi=(x_1,\cdots,x_N,y_1,\cdots,y_N,\phi)=(x,y,\phi)$$ and $$\xi'=(x_1',\cdots,x_N',y_1',\cdots,y_N',\phi')=(x',y',\phi').$$   It can be easily seen that for all $\xi\in \HH^N$,
$\xi^{-1}:=-\xi$ is the inverse of $\xi$ with respect to $\circ$.

In $\HH^N$, we define the norm
$$
|\xi|_{\HH}:=\left(\left(\sum_{i=1}^N(x_i^2+y_i^2)\right)^2 +\phi^2\right)^{\frac{1}{4}},\quad \xi=(x,y,\phi)\in \HH^N. 
$$
The mapping $d_\HH: \HH^N\times \HH^N\to [0,\infty)$ defined by
$$
d_\HH(\xi,\xi'):=\left|\xi'^{-1}\circ \xi\right|_{\HH}
$$
is a distance on $\HH^N$.  The Kor\'{a}nyi ball {with} center $\xi\in \HH^N$ and radius $R>0$ is given by 
$$
B_{\HH}(\xi,R):=\left\{\eta\in \HH^N: d_{\HH}(\xi,\eta)<R\right\}.
$$
The unit Kor\'{a}nyi ball is denoted by $B_\HH$, that is, $B_\HH:=B_\HH(0,1)$. 

Let $|\cdot|$ be  {the measure on $\HH^N$, which is Lebesgue measure on $\R^{2N+1}$}.  For all $\xi\in \HH^N$ and $R>0$, we have
$$
|B_{\HH}(\xi,R)|=|B_{\HH}(0,R)|=R^Q|B_{\HH}(0,1)|,
$$
where $Q:=2N+2$ is {the homogeneous} dimension of $\HH^N$. 

Throughout this paper, for all $\xi=(x,y,\phi)\in \HH^N\backslash\{0\}$, we shall use the notation
$$
\psi:=\frac{|x|^2+|y|^2}{|\xi|_{\HH}^2}.
$$
Let $0\leq R_1<R_2\leq \infty$ and $f\in L^1\left(B_{\HH}(0,R_2)\backslash\overline{B_{\HH}(0,R_1)}\right)$. If 
$$
f(\xi)=\psi F(\rho),\,\, \rho=|\xi|_\HH, 
$$
then {\bf (see \cite{DA})}
\begin{equation}\label{2pt1}
\int_{B_{\HH}(0,R_2)\backslash\overline{B_{\HH}(0,R_1)}}f(\xi)\,d\xi=C_N\int_{R_1}^{R_2}\rho^{2N+1}F(\rho)\,d\rho,
\end{equation}
where  
$$
{C_N:=\omega_N \int_0^\pi(\sin \theta)^N\,d\theta}
$$
and  ${\omega_N}$ is  {$2N$-surface measure of the unitary Euclidean sphere in $\RR^{2N}$.}

For all $i=1,\cdots,N$, let us consider the vector  fields
$$
X_i:=\frac{\partial }{\partial x_i}+2y_i \frac{\partial }{\partial \phi},\quad Y_i:=\frac{\partial }{\partial y_i}-2x_i \frac{\partial }{\partial \phi}.
$$
The {horizontal} gradient is defined by 
$$
\nabla_\HH:=\left(X_1,\cdots,X_N,Y_1,\cdots,Y_N\right).
$$
The Kohn-Laplacian operator is defined by 
$$
\Delta_\HH:=\sum_{i=1}^N (X_i^2+Y_i^2).
$$
Let ${f\in C^2(\Omega)}$, where ${\Omega:=B_{\HH}(0,R_2)\backslash\overline{B_{\HH}(0,R_1)}}$ and ${0\leq R_1<R_2\leq \infty}$.    If $f$ is radial, that is, $f(\xi)=F(\rho)$ ($\rho=|\xi|_\HH$), then 
$$
|\nabla_Hf(\xi)|^2=\psi|F'(\rho)|^2
$$
and
\begin{equation}\label{lap}
\Delta_\HH f(\xi)= \psi\left(F''(\rho)+\frac{2N+1}{\rho}F'(\rho)\right).
\end{equation}
In particular, we have 
$$
|\nabla_\HH \rho|^2=\psi.
$$

For $z=(x,y)\in \RR^{2N}$, let $A(z)$ be the matrix in $\RR^{2N+1}$ given by 
$$
A(z):=\left(\begin{array}{ccc}
I_N & 0_N &2y^*\\
0_N & I_N & -2x^*\\
2y &-2x & 4\displaystyle\sum_{i=1}^N\left(x_i^2+y_i^2\right)	
\end{array}
\right),
$$  
where $x=(x_1,\cdots,x_N)$, $y=(y_1,\cdots,y_N)$, $x^*$ (resp. $y^*$) is the transpose of $x$ (resp. of $y$), $I_N$ is the identity matrix in $\RR^N$ and $0_N$ is the zero matrix in $\RR^N$.  The Kohn-Laplacian operator can be written in the form
\begin{equation}\label{KLAP}
\Delta_\HH=\diver_{\RR^{2N+1}}\left(A(z)\nabla_{\RR^{2N+1}}\right),	
\end{equation}
where $\diver_{\RR^{2N+1}}$ (resp. $\nabla_{\RR^{2N+1}}$) is the divergence operator (resp. gradient operator) in $\RR^{2N+1}$.

\section{Main results}\label{sec3}

Let us first define weak solutions to \eqref{P}--\eqref{BC}. Let
$$
Q:=(0,\infty)\times \overline{B_\HH}\backslash\{0\}\quad\mbox{and}\quad  \Gamma:=(0,\infty)\times \partial B_\HH.
$$
\begin{definition}\label{def3.1}
We say that $\varphi=\varphi(t,\xi)$ is an admissible test function, if 	
\begin{itemize}
\item[{\rm{(A$_1$)}}] $\varphi\in C_{t,\xi}^{k,2}(Q)$;	
\item[{\rm{(A$_2$)}}] $\varphi\geq 0$;
\item[{\rm{(A$_3$)}}] $\supp(\varphi)\subset\subset Q$;
\item[{\rm{(A$_4$)}}] $\varphi|_{\Gamma}=0$;
\item[{\rm{(A$_5$)}}] $A\nabla_{\RR^{2N+1}}\varphi\cdot n\leq 0$ on ${{\Gamma}}$,
where  $n$ is the exterior unit normal on $\partial B_\HH$, that is, 
$${{
n(\xi):=\frac{\nabla_{\RR^{2N+1}} |\xi|_{\HH}}{\left|\nabla_{\RR^{2N+1}} |\xi|_{\HH}\right|},\quad \xi\in \partial B_\HH}}
$$
and $\cdot$ is the inner product in $\RR^{2N+1}$. 
\end{itemize}
The set of all admissible test functions is denoted by $\Phi$. 
\end{definition}

Weak solutions to \eqref{P}--\eqref{BC} are defined as follows.

\begin{definition}\label{def-ws}
Let $k\geq 1$ be an integer,  $p>1$,  $\lambda\geq -\left(\frac{Q-2}{2}\right)^2$,  $V=V(\xi)>0$ a.e. in $B_\HH$ and $f=f(\xi)\in L^1(\partial B_\HH)$. We say that $u$ is a weak solution to  \eqref{P}--\eqref{BC}, if $u\in L^p_{\loc}(Q,V\psi \,d\xi\,dt)\cap L^1_{\loc}(Q,\psi\,d\xi\,dt)$ and 
\begin{equation}\label{ws}
\begin{aligned}
&\int_Q |u|^p\varphi V\psi\,d\xi\,dt-\int_{\Gamma} f A\nabla_{\RR^{2N+1}}\varphi\cdot n\,{{dH_{2N}}} \,dt\\
&\leq \int_Q u\left((-1)^k\frac{\partial^k\varphi}{\partial t^k}-\frac{1}{\psi}\Delta_\HH\varphi+\frac{\lambda}{|\xi|_\HH^2}\varphi\right)\psi\,d\xi\,dt
\end{aligned}
\end{equation}
for every $\varphi\in \Phi$, where ${{dH_{2N}}}$ {is the $2N$-dimensional surface measure in $\RR^{2N+1}$.} 
\end{definition}

Notice that, if $u\in  C_{t,\xi}^{k,2}(Q)$ is a (classical) solution to \eqref{P}--\eqref{BC}, then $u$ is a weak solution in the sense of Definition \ref{def-ws}. Namely, multiplying \eqref{P} by $\psi\varphi$, where $\varphi\in \Phi$, integrating by parts over $Q$, using \eqref{KLAP},  {the Euclidean version of the divergence theorem} and the boundary condition \eqref{BC}, we obtain \eqref{ws}.

For $\lambda\geq -\left(\frac{Q-2}{2}\right)^2$, we introduce the parameters
$$
\alpha^\pm:=-\frac{Q-2}{2}\pm \sqrt{\lambda+\left(\frac{Q-2}{2}\right)^2}.	
$$
Let $\sigma_\lambda$ be the function defined in $(0,\infty)$ by 
$$
\sigma_\lambda(s):=\left\{\begin{array}{llll}
s^{\alpha^-}-s^{\alpha^+}  &\mbox{if}& \lambda>-\left(\frac{Q-2}{2}\right)^2,\\[2pt]
-s^{\alpha^-} \ln s &\mbox{if}& \lambda=-\left(\frac{Q-2}{2}\right)^2.	
\end{array}
\right.
$$
We also introduce the radial function $K$ defined  by 
\begin{equation}\label{K-f}
K(\xi):=\sigma_\lambda(|\xi|_\HH),\quad \xi\in B_\HH\backslash\{0\}.	
\end{equation} 
Let $L^{1,+}(\partial B_\HH)$ be the set of functions defined by
$$
L^{1,+}(\partial B_\HH):=\left\{h\in L^1(\partial B_\HH): \int_{\partial B_\HH}h(\xi) \frac{\psi}{\left|\nabla_{\RR^{2N+1}}|\xi|_\HH\right|}\,{{dH_{2N}}}>0\right\}.
$$

Our first main result is stated in the following theorem.

\begin{theorem}\label{T1}
Let $k\geq 1$ be an integer,  $p>1$,  $\lambda\geq -\left(\frac{Q-2}{2}\right)^2$,  $V=V(\xi)>0$ a.e. in $B_\HH$ and $V^{\frac{-1}{p-1}}\in L^1_{\loc}(B_\HH\backslash\{0\},K\psi\,d\xi)$.  Assume that 
$f\in L^{1,+}(\partial B_\HH)$. If 
\begin{equation}\label{cd1-bl}
\liminf_{R\to \infty}R^{\frac{2p}{p-1}} \int_{B_\HH\left(0,\frac{1}{R}\right)\backslash\overline{B_\HH\left(0,\frac{1}{2R}\right)}} 	V^{\frac{-1}{p-1}}K\psi\,d\xi=0,
\end{equation}
then \eqref{P}--\eqref{BC} admits no weak solution. 	
\end{theorem}

We now consider the case when
\begin{equation}\label{V-f}
V(\xi)=|\xi|_{\mathbb{H}}^a,\quad \xi\in B_\HH\backslash\{0\}, 	
\end{equation}
where  $a\in \RR$ is a  constant, and $V^{\frac{-1}{p-1}}\in L^1_{\loc}(B_\HH\backslash\{0\},K\psi\,d\xi)$ for all $a\in \RR$. In this case, we have the following result.

\begin{theorem}\label{T2}
Let $k\geq 1$ be an integer,  $p>1$, $\lambda\geq -\left(\frac{Q-2}{2}\right)^2$, and $V$ be the function  given by \eqref{V-f}.
\begin{itemize}
\item[{\rm{(I)}}] 	If $f\in L^{1,+}(\partial B_\HH)$ and 
\begin{equation}\label{cd-bl-V}
\left\{\begin{array}{llll}
(Q-2+\alpha^-)p\geq Q +a +\alpha^- &\mbox{if}& \lambda> -\left(\frac{Q-2}{2}\right)^2,\\
(Q-2+\alpha^-)p> Q +a +\alpha^-	 &\mbox{if}& \lambda= -\left(\frac{Q-2}{2}\right)^2,
\end{array}
\right.	
\end{equation}
then \eqref{P}--\eqref{BC} admits no weak solution.
\item[{\rm{(II)}}] If 
\begin{equation}\label{cd-ex-V}
(Q-2+\alpha^-)p< Q +a +\alpha^-,	
\end{equation}
then \eqref{P}--\eqref{BC} admits stationary solutions $u\in C^\infty(B_\HH\backslash\{0\})$ for some  $f\in L^{1,+}(\partial B_\HH)$.
\end{itemize}
\end{theorem}

\begin{remark}\label{RK1}
Observe that the existence and nonexistence results provided by Theorems \ref{T1} and \ref{T2}  are independent of the order of the time-derivative $k$. So the obtained results hold true for both parabolic ($k=1$) and 	hyperbolic ($k=2$) cases. 
\end{remark}

Clearly, Theorem \ref{T2} yields existence and nonexistence results for the stationary problem
\begin{equation}\label{PS}
\frac{1}{\psi}\Delta_\HH u+\frac{\lambda}{|\xi|_\HH^2}u\geq |\xi|_\HH^a |u|^p\quad \mbox{in } B_\HH	
\end{equation}
under the boundary condition \eqref{BC}. Namely, we have the following result.

\begin{corollary}
Let $p>1$, $\lambda\geq -\left(\frac{Q-2}{2}\right)^2$ and $a\in \RR$.
\begin{itemize}
\item[{\rm{(I)}}] 	If $f\in L^{1,+}(\partial B_\HH)$ and \eqref{cd-bl-V} holds,
then \eqref{PS} under the boundary condition \eqref{BC} admits no weak solution.
\item[{\rm{(II)}}] If \eqref{cd-ex-V} holds,  then  \eqref{PS} under the boundary condition \eqref{BC}  admits solutions $u\in C^\infty(B_\HH\backslash\{0\})$ for some  $f\in L^{1,+}(\partial B_\HH)$.
\end{itemize}
\end{corollary}

Let us study some special cases of Theorem \ref{T2}. Just before, we introduce below three kinds of critical exponents for \eqref{P}--\eqref{BC}, where $V$ is defined by \eqref{V-f}. Namely, we shall deduce from Theorem \ref{T2} that \eqref{P}--\eqref{BC} admits three different critical behaviors according to the value of  the parameter $\lambda$.

\begin{definition}[Critical exponent of the first kind]\label{def-cr}
Let $p_{cr}>1$. We say that $p_{cr}$ is a critical exponent of the first kind for  \eqref{P}--\eqref{BC}, where $V$ is defined by \eqref{V-f}, if
\begin{itemize}
\item[{\rm{(i)}}] for all $1<p< p_{cr}$,	the problem \eqref{P}--\eqref{BC} admits no weak solution,  provided $f\in L^{1,+}(\partial B_\HH)$;
\item[{\rm{(ii)}}] for all $p>p_{cr}$, the problem \eqref{P}--\eqref{BC} admits weak solutions for some $f\in L^{1,+}(\partial B_\HH)$.
\end{itemize}
\end{definition}

\begin{definition}[Critical exponent of the second kind]\label{def2-cr}
Let $p_{cr}>1$. We say that $p_{cr}$ is a critical exponent of the second kind for  \eqref{P}--\eqref{BC}, where $V$ is defined by \eqref{V-f}, if
\begin{itemize}
\item[{\rm{(i)}}] for all $p> p_{cr}$, the problem \eqref{P}--\eqref{BC} admits no weak solution, provided $f\in L^{1,+}(\partial B_\HH)$;
\item[{\rm{(ii)}}] for all $1<p<p_{cr}$, the problem \eqref{P}--\eqref{BC} admits weak solutions for some $f\in L^{1,+}(\partial B_\HH)$.
\end{itemize}
\end{definition}

\begin{definition}[Critical exponent of the third kind]\label{def2-cr}
We say that $a^*\in \RR$ is a critical exponent of the third kind for  \eqref{P}--\eqref{BC}, where $V$ is defined by \eqref{V-f}, if 	
\begin{itemize}
\item[{\rm{(i)}}] for all $p>1$, the problem \eqref{P}--\eqref{BC} admits no weak solution, provided $f\in L^{1,+}(\partial B_\HH)$ and $a\leq a^*$;
\item[{\rm{(ii)}}] for all $p>1$, the problem \eqref{P}--\eqref{BC} admits weak solutions for some $f\in L^{1,+}(\partial B_\HH)$, provided $a>a^*$.
\end{itemize}
\end{definition}

Assume first that
$$
\lambda=-\left(\frac{Q-2}{2}\right)^2.
$$
In this case, one has
$$
Q-2+\alpha^-=\frac{Q-2}{2}>0$$ and $$Q+a+\alpha^-=\frac{Q+2}{2}+a.
$$
Hence, from Theorem \ref{T2}, we deduce the following result.

\begin{corollary}\label{CR3.9}
Let $k\geq 1$ be an integer,	 $\lambda=-\left(\frac{Q-2}{2}\right)^2$ and $V$ be the function  given by \eqref{V-f}.
\begin{itemize}
\item[{\rm{(I)}}] Let $f\in L^{1,+}(\partial B_\HH)$.
\begin{itemize}
\item[{\rm{(i)}}] If $a\leq -2$, then for all $p>1$, the problem \eqref{P}--\eqref{BC} admits no weak solution.
\item[{\rm{(ii)}}] If $a>-2$, then for all
$$
p> 1+\frac{2(a+2)}{Q-2},
$$
the problem \eqref{P}--\eqref{BC} admits no weak solution.
\end{itemize}
\item[{\rm{(II)}}] If $a>-2$ and
$$
1<p<1+\frac{2(a+2)}{Q-2},
$$
then the problem \eqref{P}--\eqref{BC} admits stationary solutions $u\in C^\infty(B_\HH\backslash\{0\})$ for some  $f\in L^{1,+}(\partial B_\HH)$.
\end{itemize}
\end{corollary}

\begin{remark}
At this moment, in the case $$\lambda=-\left(\frac{Q-2}{2}\right)^2$$ and $$p= 1+\frac{2(a+2)}{Q-2},$$ we do not know whether we have existence or nonexistence of weak solutions to  \eqref{P}--\eqref{BC}, where 	$V$ is the function  given by \eqref{V-f}. This question is left  open.
\end{remark}

Assume now that
$$
-\left(\frac{Q-2}{2}\right)^2<\lambda<0.
$$
In this case, one has
$$
Q-2+\alpha^->0.
$$
Hence, from Theorem \ref{T2}, we deduce the following result.

\begin{corollary}\label{CR3.10}
Let $k\geq 1$ be an integer,	 $-\left(\frac{Q-2}{2}\right)^2<\lambda<0$ and $V$ be the function  given by \eqref{V-f}.	\begin{itemize}
\item[{\rm{(I)}}] Let $f\in L^{1,+}(\partial B_\HH)$.
\begin{itemize}
\item[{\rm{(i)}}] If $a\leq -2$, then for all $p>1$, the problem \eqref{P}--\eqref{BC} admits no weak solution.
\item[{\rm{(ii)}}] If $a>-2$, then for all
$$
p\geq  1+\frac{a+2}{Q-2+\alpha^-},
$$
the problem \eqref{P}--\eqref{BC} admits no weak solution.
\end{itemize}
\item[{\rm{(II)}}] If $a>-2$ and
$$
1<p<1+\frac{a+2}{Q-2+\alpha^-},
$$
then the problem \eqref{P}--\eqref{BC} admits stationary solutions $u\in C^\infty(B_\HH\backslash\{0\})$ for some  $f\in L^{1,+}(\partial B_\HH)$.
\end{itemize}
\end{corollary}

\begin{remark}
From Corollaries \ref{CR3.9} and \ref{CR3.10}, we deduce that, if  $$-\left(\frac{Q-2}{2}\right)^2\leq \lambda<0$$ and 	$V$ is the function  given by \eqref{V-f}, where $a>-2$, then
\begin{equation}\label{p-jl}
p_{cr}(Q,\lambda,a):= 1+\frac{a+2}{Q-2+\alpha^+}
\end{equation}
is a critical exponent of the second kind for \eqref{P}--\eqref{BC}.   	
\end{remark}

We next assume that $\lambda=0$.  In this case, one has
$$
Q-2+\alpha^-=0$$ and $$Q+a+\alpha^-= a+2.
$$
Hence, from Theorem \ref{T2}, we deduce the following result.

\begin{corollary}\label{CR3.13}
Let $k\geq 1$ be an integer,	 $\lambda=0$ and $V$ be the function  given by \eqref{V-f}.
\begin{itemize}
\item[{\rm{(I)}}] If $a\leq -2$ 	and $f\in L^{1,+}(\partial B_\HH)$, then for all $p>1$, the problem \eqref{P}--\eqref{BC} admits no weak solution.
\item[{\rm{(II)}}] If $a>-2$, then for all $p>1$, the problem \eqref{P}--\eqref{BC} admits stationary solutions $u\in C^\infty(B_\HH\backslash\{0\})$ for some  $f\in L^{1,+}(\partial B_\HH)$.
\end{itemize}
\end{corollary}

\begin{remark}
From Corollary \ref{CR3.13}, if $\lambda=0$ and 	$V$ is the function  given by \eqref{V-f}, then $$a^*:=-2$$ is a critical exponent of the third kind for the problem \eqref{P}--\eqref{BC}.
\end{remark}

We finally assume that $\lambda>0$. In this case, one has
$$
Q-2+\alpha^-<0.
$$
Then  Theorem \ref{T2} yields the following result.

\begin{corollary}\label{CR3.15}
Let $k\geq 1$ be an integer,	 $\lambda>0$ and $V$ be the function  given by \eqref{V-f}.
\begin{itemize}
\item[{\rm{(I)}}] Let $f\in L^{1,+}(\partial B_\HH)$.

\begin{itemize}
\item[{\rm{(i)}}] If $a\geq -2$, then for all $p>1$, the problem \eqref{P}--\eqref{BC} admits no weak solution.
\item[{\rm{(ii)}}] If $a<-2$, then for all
$$
1<p\leq  1+\frac{a+2}{Q-2+\alpha^+},
$$
the problem \eqref{P}--\eqref{BC} admits no weak solution.
\end{itemize}
\item[{\rm{(II)}}] If $a<-2$ and
$$
p>1+\frac{a+2}{Q-2+\alpha^+},
$$
then the problem \eqref{P}--\eqref{BC} admits stationary solutions $u\in C^\infty(B_\HH\backslash\{0\})$ for some  $f\in L^{1,+}(\partial B_\HH)$.
\end{itemize}
\end{corollary}

\begin{remark}
From Corollary \ref{CR3.15}, if $\lambda>0$ and 	$V$ is the function  given by \eqref{V-f}, then the real number $$p_{cr}(Q,\lambda,a):= 1+\frac{a+2}{Q-2+\alpha^+}$$ given by \eqref{p-jl} is a critical exponent of the first  kind for \eqref{P}--\eqref{BC}.
\end{remark}

\section{Auxiliary results}\label{sec4}
Let $k\geq 1$ be an integer,  $p>1$,  $\lambda\geq -\left(\frac{Q-2}{2}\right)^2$,  $V=V(\xi)>0$ a.e. in $B_\HH$ and $f=f(\xi)\in L^1(\partial B_\HH)$.

\subsection{Admissible test functions}
Two kinds of admissible test functions will be  introduced in this subsection.

Let $K$ be the radial function defined by \eqref{K-f}. We collect below some properties of the function $K$.

\begin{lemma}\label{L4.1}
The function $K$ satisfies the following properties:
\begin{itemize}
\item[{\rm{(i)}}] $K\in C^2(\overline{B_\HH}\backslash\{0\})$;	
\item[{\rm{(ii)}}] $K\geq 0$;
\item[{\rm{(iii)}}] $-\displaystyle\frac{1}{\psi}\Delta_\HH K+\displaystyle\frac{\lambda}{|\xi|_\HH^2}K=0$ in $B_\HH\backslash\{0\}$;
\item[{\rm{(iv)}}] $K|_{\partial B_\HH}=0$ and
\begin{equation}\label{nder}
A\nabla_{\RR^{2N+1}}K\cdot n|_{\partial B_\HH}=\left\{\begin{array}{llll}
-2\sqrt{\lambda+\left(\frac{Q-2}{2}\right)^2} \displaystyle \frac{\psi}{|\nabla_{\RR^{2N+1}}|\xi|_\HH|}&\mbox{if}&  \lambda> -\left(\frac{Q-2}{2}\right)^2,\\
-\displaystyle\frac{\psi}{|\nabla_{\RR^{2N+1}}|\xi|_\HH|}&\mbox{if}& \lambda= -\left(\frac{Q-2}{2}\right)^2.
\end{array}
\right.
\end{equation}
\end{itemize}
\end{lemma}

\begin{proof}
(i), (ii) and (iv) follow immediately from the definition  of $K$. On the other hand, since $K$ is a radial function, making use of \eqref{lap}, we obtain (iii).		
\end{proof}

Le $\vartheta,\zeta\in C^\infty([0,\infty))$  be two cut-off function satisfying the following properties:
\begin{equation}\label{v-ppts}
0\leq \vartheta\leq 1,\,\, \supp(\vartheta)\subset\subset (0,1)
\end{equation}
and
\begin{equation}\label{ppts-zeta}
\zeta\mbox{ is nondecreasing},\,\, \zeta= 0\mbox{ in } \left[0,\frac{1}{2}\right],\,\,\zeta= 1\mbox{ in } [1,\infty).
\end{equation}
For $\iota,T,R\gg 1$, let
\begin{equation}\label{betaT-f}
\beta_T(t):=\vartheta^\iota\left(\frac{t}{T}\right),\quad t\geq 0,
\end{equation}
and
\begin{equation}\label{gamaR-f}
\gamma_R(\xi):=K(\xi) \zeta^\iota\left(R|\xi|_\HH\right)=\sigma_\lambda(\rho)\zeta^\iota\left(R\rho\right),\quad \xi\in \overline{B_\HH}\backslash\{0\},\,\, \rho=|\xi|_\HH.
\end{equation}
We introduce test functions of the form
\begin{equation}\label{testf1}
\varphi(t,x):=\beta_T(t)\gamma_R(\xi),\quad (t,\xi)\in Q.	
\end{equation}

\begin{lemma}\label{L4.2}
The function $\varphi$ defined by \eqref{testf1} belongs to $\Phi$.	
\end{lemma}

\begin{proof}
(A$_1$)--(A$_3$) follow immediately from Lemma \ref{L4.1} (i)-(ii), \eqref{v-ppts}, \eqref{ppts-zeta}, \eqref{betaT-f}, \eqref{gamaR-f} and \eqref{testf1}. Furthermore, by Lemma \ref{L4.1} (iv)  ($K|_{\partial B_\HH}=0$) and \eqref{testf1}, we obtain (A$_4$). On the other hand, by \eqref{ppts-zeta} and \eqref{gamaR-f}, we have
$$
\gamma_R(\xi)=K(\xi),\quad \xi\in \overline{B_\HH}\backslash B_\HH\left(0,\frac{1}{R}\right),
$$
which implies by \eqref{nder} and \eqref{testf1} that (notice that $\beta_T\geq 0$)
$$
\begin{aligned}
A\nabla_{\RR^{2N+1}}\varphi(t,\xi)\cdot n(\xi)&=\beta_T(t)A\nabla_{\RR^{2N+1}}\gamma_R(\xi)\cdot n(\xi)\\
&=\beta_T(t)A\nabla_{\RR^{2N+1}}K(\xi)\cdot n(\xi)\leq 0
\end{aligned}
$$
for all $(t,\xi)\in \Gamma$. This shows that $\varphi$ satisfies (A$_5$). Consequently,  $\varphi$  is an admissible test function in the sense of Definition \ref{def3.1}.
\end{proof}

Let $\ell: \mathbb{R}\to [0,1]$ be a nondecreasing smooth function satisfying
\begin{equation}\label{ell-f}
\ell(s):=\left\{\begin{array}{llll}
0 &\mbox{if}& s\leq 0,	\\{}\\
1 &\mbox{if}& s\geq \frac{1}{2}.
\end{array}
\right.
\end{equation}
For $\iota,T,R\gg 1$, let  $\beta_T$ be the function defined by \eqref{betaT-f} and
\begin{equation}\label{muRR}
\begin{split}\mu_R(\xi):&=K(\xi) \ell^\iota\left(1+\frac{\ln |\xi|_\HH}{\ln R}\right)\\&=K(\rho)\ell^\iota\left(1+\frac{\ln \rho}{\ln R}\right),\quad \xi\in \overline{B_\HH}\backslash\{0\},\,\, \rho=|\xi|_\HH.\end{split}
\end{equation}
We introduce test functions of the form
\begin{equation}\label{testf2}
\varphi(t,\xi):=\beta_T(t)\mu_R(\xi),\,\quad (t,\xi)\in Q.		
\end{equation}

\begin{lemma}\label{L4.3}
The function $\varphi$ defined by \eqref{testf2} belongs to $\Phi$.	
\end{lemma}

\begin{proof}
The proof is similar to that of Lemma \ref{L4.2}. Namely, (A$_1$)--(A$_3$) follow immediately from Lemma \ref{L4.1} (i)-(ii), \eqref{v-ppts},\eqref{betaT-f}, \eqref{ell-f},  \eqref{muRR} and \eqref{testf2}. 	Furthermore, by Lemma \ref{L4.1} (iv)  ($K|_{\partial B_\HH}=0$) and \eqref{testf2}, we obtain (A$_4$). On the other hand, by \eqref{ell-f} and \eqref{muRR}, we have
\begin{equation}\label{jalh}
\mu_R(\xi)=K(\xi),\quad \xi\in \overline{B_\HH}\backslash B_\HH\left(0,\frac{1}{\sqrt R}\right).
\end{equation}
Then  (A$_5$) follows from  \eqref{nder}, \eqref{testf2} and \eqref{jalh}.
\end{proof}

\subsection{A priori estimate}
For all $\varphi\in \Phi$, we consider the integral terms
\begin{equation}\label{J1}
J_1(\varphi):= 	\int_Q \varphi^{\frac{-1}{p-1}}
\left|\frac{\partial^k\varphi}{\partial t^k}\right|^{\frac{p}{p-1}}V^{\frac{-1}{p-1}}\psi\,d\xi\,dt
\end{equation}
and
\begin{equation}\label{J2}
J_2(\varphi):=\int_Q  \varphi^{\frac{-1}{p-1}} \left|-\Delta_\HH \varphi+\frac{\lambda}{|\xi|_\HH^2}\psi\varphi\right|^{\frac{p}{p-1}}V^{\frac{-1}{p-1}}\psi^{\frac{-1}{p-1}}\,d\xi\,dt.
\end{equation}
Let $\varphi\in \Phi$ be such that $J_i(\varphi)<\infty$ for all $i=1,2$. Assume that
$u$ is a weak solution to  \eqref{P}--\eqref{BC}. By \eqref{ws}, we obtain
\begin{equation}\label{S1-L4.4}
\begin{aligned}
&\int_Q |u|^p\varphi V\psi\,d\xi\,dt-\int_{\Gamma} f A\nabla_{\RR^{2N+1}}\varphi\cdot n\,d\xi \,dt\\
&\leq  \int_{Q} |u|\left|-\Delta_\HH \varphi+\frac{\lambda}{|\xi|_\HH^2}\psi\varphi \right|\,d\xi\,dt + \int_{Q} |u| \psi \left|\frac{\partial ^k\varphi}{\partial t^k}\right|\,d\xi\,dt.
\end{aligned}
\end{equation}
Making use of Young's inequality, we obtain
\begin{equation}\label{S2-L4.4}
\begin{aligned}
\int_{Q} |u| \psi \left|\frac{\partial ^k\varphi}{\partial t^k}\right|\,d\xi\,dt&=\int_{Q} |u| \psi^{\frac{1}{p}}\varphi^{\frac{1}{p}}V^{\frac{1}{p}} \psi^{\frac{p-1}{p}}\varphi^{\frac{-1}{p}} V^{\frac{-1}{p}}\left|\frac{\partial ^k\varphi}{\partial t^k}\right|\,d\xi\,dt	\\
&\leq \frac{1}{2}\int_Q |u|^p\varphi V\psi\,d\xi\,dt+C J_1(\varphi).
\end{aligned}
\end{equation}
Similarly, we have
\begin{equation}\label{S3-L4.4}
\begin{aligned}
&\int_{Q} |u|\left|-\Delta_\HH \varphi+\frac{\lambda}{|\xi|_\HH^2}\psi\varphi \right|\,d\xi\,dt\\
&=\int_{Q} |u| \psi^{\frac{1}{p}}\varphi^{\frac{1}{p}}V^{\frac{1}{p}}  \psi^{\frac{-1}{p}}\varphi^{\frac{-1}{p}} V^{\frac{-1}{p}} \left|-\Delta_\HH \varphi+\frac{\lambda}{|\xi|_\HH^2}\psi\varphi \right|\,d\xi\,dt	\\
&\leq \frac{1}{2}\int_Q |u|^p\varphi V\psi\,d\xi\,dt+C J_2(\varphi).
\end{aligned}
\end{equation}
Then, in view of \eqref{S1-L4.4}, \eqref{S2-L4.4} and \eqref{S3-L4.4}, we obtain the following a priori estimate.

\begin{lemma}\label{L4.4}
If $u$ is a weak solution to  \eqref{P}--\eqref{BC}, then
\begin{equation}\label{apest}
-\int_{\Gamma} f A\nabla_{\RR^{2N+1}}\varphi\cdot n\,{{dH_{2N}}} \,dt
\leq C \sum_{i=1}^2 J_i(\varphi)
\end{equation}
for every $\varphi\in \Phi$, provided $J_i(\varphi)<\infty$ for all $i=1,2$.  
\end{lemma}

\subsection{Estimates of $J_i(\varphi)$}

\begin{lemma}\label{L4.5}
We have
\begin{equation}\label{est1-L.4.5}
\int_0^\infty 	\beta_T(t)\,dt \leq   T
\end{equation}
and
\begin{equation}
\label{est2-L.4.5} \int_0^\infty  \beta_T^{\frac{-1}{p-1}}\left|\beta_T^{(k)}\right|^{\frac{p}{p-1}}\,dt \leq  CT^{1-\frac{kp}{p-1}},
\end{equation}
where $\beta_T$ is the function defined by \eqref{betaT-f}.
\end{lemma}

\begin{proof}
By \eqref{v-ppts} (namely $0\leq \vartheta\leq 1$), we have
$$
\begin{aligned}
\int_0^\infty 	\beta_T(t)\,dt  &=\int_0^\infty 	\vartheta^\iota\left(\frac{t}{T}\right)\,dt \\
&= \int_0^{T} 		\vartheta^\iota\left(\frac{t}{T}\right)\,dt \\
&\leq T,
\end{aligned}
$$	
which proves \eqref{est1-L.4.5}. Furthermore, we have
\begin{equation}\label{S1-L4.5}
 \int_0^\infty  \beta_T^{\frac{-1}{p-1}}\left|\beta_T^{(k)}\right|^{\frac{p}{p-1}}\,dt= \int_0^{T}  \vartheta^{\frac{-\iota}{p-1}}\left(\frac{t}{T}\right)\left|\beta_T^{(k)}\right|^{\frac{p}{p-1}}\,dt.
\end{equation}
On the other hand, for all $0<t<T$, we have
$$
\begin{aligned}
\left|\beta_T^{(k)}(t)\right|& =\left|\left[\vartheta^\iota\left(\frac{t}{T}\right)\right]^{(k)}\right|\\
&\leq C T^{-k}\vartheta^{\iota-k}\left(\frac{t}{T}\right),
\end{aligned}
$$
which implies by \eqref{S1-L4.5} that
$$
\begin{aligned}
 \int_0^\infty  \beta_T^{\frac{-1}{p-1}}\left|\beta_T^{(k)}\right|^{\frac{p}{p-1}}\,dt&\leq C T^{\frac{-kp}{p-1}}	\int_0^{T}\vartheta^{\iota-\frac{kp}{p-1}}\left(\frac{t}{T}\right)\,dt\\
 &\leq C T^{1-\frac{kp}{p-1}},
\end{aligned}
$$
which proves \eqref{est2-L.4.5}.
\end{proof}

\begin{lemma}\label{L4.6}
Let $V^{\frac{-1}{p-1}}\in L^1_{\loc}(B_\HH\backslash\{0\},K\psi\,d\xi)$. Then
\begin{equation}\label{est-L4.6}
\int_{B_\HH} \gamma_RV^{\frac{-1}{p-1}}\psi\,d\xi \leq \int_{B_H\backslash \overline{B_\HH\left(0,\frac{1}{2R}\right)}}V^{\frac{-1}{p-1}}K\psi\,d\xi,
\end{equation}
where $\gamma_R$ is the function defined by \eqref{gamaR-f}.
\end{lemma}

\begin{proof}
By \eqref{ppts-zeta} and \eqref{gamaR-f}, we have
$$
\begin{aligned}
\int_{B_\HH} \gamma_R V^{\frac{-1}{p-1}}\psi\,d\xi& =\int_{B_\HH}K(\xi) \zeta^\iota\left(R|\xi|_\HH\right)V^{\frac{-1}{p-1}}\psi\,d\xi \\
&=\int_{B_H\backslash \overline{B_\HH\left(0,\frac{1}{2R}\right)}}K(\xi)\zeta^\iota\left(R|\xi|_\HH\right)V^{\frac{-1}{p-1}}\psi\,d\xi\\
&\leq \int_{B_H\backslash \overline{B_\HH\left(0,\frac{1}{2R}\right)}}V^{\frac{-1}{p-1}}K\psi\,d\xi,
\end{aligned}
$$
which proves \eqref{est-L4.6}. 	
\end{proof}

\begin{lemma}\label{L4.7}
Let $V^{\frac{-1}{p-1}}\in L^1_{\loc}(B_\HH\backslash\{0\},K\psi\,d\xi)$. Then
\begin{equation}\label{est-L4.7}
\int_{B_\HH} \mu_RV^{\frac{-1}{p-1}}\psi\,d\xi \leq \int_{B_H\backslash \overline{B_\HH\left(0,\frac{1}{R}\right)}}V^{\frac{-1}{p-1}}K\psi\,d\xi
\end{equation}
where $\mu_R$ is the function defined by \eqref{muRR}.
\end{lemma}

\begin{proof}
By \eqref{ell-f} and \eqref{muRR}, we have	
$$
\begin{aligned}
\int_{B_\HH} \mu_RV^{\frac{-1}{p-1}}\psi\,d\xi &= \int_{B_\HH}K(\xi) \ell^\iota\left(1+\frac{\ln |\xi|_\HH}{\ln R}\right)V^{\frac{-1}{p-1}}\psi\,d\xi\\
&=\int_{B_H\backslash \overline{B_\HH\left(0,\frac{1}{R}\right)}}K(\xi) \ell^\iota\left(1+\frac{\ln |\xi|_\HH}{\ln R}\right)V^{\frac{-1}{p-1}}\psi\,d\xi\\
&\leq \int_{B_H\backslash \overline{B_\HH\left(0,\frac{1}{R}\right)}} V^{\frac{-1}{p-1}}K\psi\,d\xi,
	\end{aligned}
$$
which proves \eqref{est-L4.7}.
\end{proof}

\begin{lemma}\label{L4.8}
Let $V^{\frac{-1}{p-1}}\in L^1_{\loc}(B_\HH\backslash\{0\},K\psi\,d\xi)$. Then
\begin{equation}\label{est-L4.8}
\begin{aligned}
&\int_{B_\HH}	\gamma_R^{\frac{-1}{p-1}} \left|-\Delta_\HH \gamma_R+\frac{\lambda}{|\xi|_\HH^2}\psi\gamma_R\right|^{\frac{p}{p-1}}V^{\frac{-1}{p-1}}\psi^{\frac{-1}{p-1}}\,d\xi\\
&\leq CR^{\frac{2p}{p-1}}\int_{B_\HH\left(0,\frac{1}{R}\right)\backslash \overline{B_\HH\left(0,\frac{1}{2R}\right)}}	V^{\frac{-1}{p-1}}K \psi \,d\xi.
\end{aligned}
\end{equation}
\end{lemma}

\begin{proof}
By \eqref{lap},  \eqref{K-f} and \eqref{gamaR-f}, for all $\xi\in B_\HH\backslash\{0\}$, we have (with $\rho=|\xi|_\HH$)
$$
\begin{aligned}
& -\Delta_\HH \gamma_R(\xi)+\frac{\lambda}{|\xi|_\HH^2}\psi\gamma_R(\xi)\\
&= -\zeta^\iota\left(R\rho\right)\Delta_\HH K(\xi)-K(\xi)\Delta_\HH \left[	\zeta^\iota\left(R\rho\right)\right]-2\psi  \sigma_\lambda'(\rho)\left[\zeta^\iota\left(R\rho\right)\right]'+\frac{\lambda}{\rho^2}\psi K(\xi)\zeta^\iota\left(R\rho\right)\\
&=\zeta^\iota\left(R\rho\right)\left(-\Delta_\HH K(\xi)+\frac{\lambda}{\rho^2}\psi K(\xi)\right)-\sigma_\lambda(\rho)\Delta_\HH \left[	\zeta^\iota\left(R\rho\right)\right]-2\psi  \sigma_\lambda'(\rho)\left[\zeta^\iota\left(R\rho\right)\right]',
\end{aligned}
$$	
which implies by Lemma \ref{L4.1} (iii) that
\begin{equation}\label{S1-L4.8}
-\Delta_\HH \gamma_R(\xi)+\frac{\lambda}{|\xi|_\HH^2}\psi\gamma_R(\xi)=	-\sigma_\lambda(\rho)\Delta_\HH \left[	\zeta^\iota\left(R\rho\right)\right]-2\psi  \sigma_\lambda'(\rho)\left[\zeta^\iota\left(R\rho\right)\right]'.
\end{equation}
Then, from  \eqref{ppts-zeta} and \eqref{S1-L4.8},  we deduce that
\begin{equation}\label{S2-L4.8}
\begin{aligned}
&\int_{B_\HH}	\gamma_R^{\frac{-1}{p-1}} \left|-\Delta_\HH \gamma_R+\frac{\lambda}{|\xi|_\HH^2}\psi\gamma_R\right|^{\frac{p}{p-1}}V^{\frac{-1}{p-1}}\psi^{\frac{-1}{p-1}}\,d\xi\\
&=\int_{B_\HH\left(0,\frac{1}{R}\right)\backslash \overline{B_\HH\left(0,\frac{1}{2R}\right)}}	\gamma_R^{\frac{-1}{p-1}} \left|-\Delta_\HH \gamma_R+\frac{\lambda}{|\xi|_\HH^2}\psi\gamma_R\right|^{\frac{p}{p-1}}V^{\frac{-1}{p-1}}\psi^{\frac{-1}{p-1}}\,d\xi.
\end{aligned}	
\end{equation}
Furthermore, by \eqref{lap}, \eqref{K-f} and \eqref{S1-L4.8}, for all $\xi\in B_\HH\left(0,\frac{1}{R}\right)\backslash \overline{B_\HH\left(0,\frac{1}{2R}\right)}$, we obtain
$$
\begin{aligned}
\left|-\Delta_\HH \gamma_R(\xi)+\frac{\lambda}{|\xi|_\HH^2}\psi\gamma_R(\xi)\right|
&\leq K(\xi) \big|\Delta_\HH \left[\zeta^\iota\left(R\rho\right)\right]\big|+C \psi K(\xi) \frac{|\sigma_\lambda'(\rho)|}{\sigma_\lambda(\rho)}\big|\left[\zeta^\iota\left(R\rho\right)\right]'\big| \\
&\leq K(\xi)\left(\big|\Delta_\HH \left[\zeta^\iota\left(R\rho\right)\right]\big|+C R\psi\big|\left[\zeta^\iota\left(R\rho\right)\right]'\big|\right)\\
&\leq CR^2K(\xi)\psi \zeta^{\iota-2}\left(R\rho\right),
\end{aligned}
$$
which implies (recall that $\iota\gg 1$ and $0\leq \zeta\leq 1$) that
\begin{equation}\label{S3-L4.8}
\begin{aligned}
\gamma_R^{\frac{-1}{p-1}} \left|-\Delta_\HH \gamma_R+\frac{\lambda}{|\xi|_\HH^2}\psi\gamma_R\right|^{\frac{p}{p-1}}V^{\frac{-1}{p-1}}\psi^{\frac{-1}{p-1}} &\leq CR^{\frac{2p}{p-1}}V^{\frac{-1}{p-1}}K\psi \zeta^{\iota-\frac{2p}{p-1}}\left(R\rho\right)\\
&\leq CR^{\frac{2p}{p-1}}V^{\frac{-1}{p-1}}K(\xi)\psi.
\end{aligned}
\end{equation}
Finally, \eqref{est-L4.8} follows from \eqref{S2-L4.8} and \eqref{S3-L4.8}.
\end{proof}

\begin{lemma}\label{L4.9}
Let $V^{\frac{-1}{p-1}}\in L^1_{\loc}(B_\HH\backslash\{0\},K\psi\,d\xi)$. Then	
\begin{equation}\label{est-L4.9}
\begin{aligned}
&\int_{B_\HH}	\mu_R^{\frac{-1}{p-1}} \left|-\Delta_\HH \mu_R+\frac{\lambda}{|\xi|_\HH^2}\psi\mu_R\right|^{\frac{p}{p-1}}V^{\frac{-1}{p-1}}\psi^{\frac{-1}{p-1}}\,d\xi\\
&\leq C(\ln R)^{\frac{-p}{p-1}}\int_{B_\HH\left(0,\frac{1}{\sqrt R}\right)\backslash \overline{B_\HH\left(0,\frac{1}{R}\right)}}	V^{\frac{-1}{p-1}}K \psi |\xi|_\HH^{\frac{-2p}{p-1}}\,d\xi.
\end{aligned}
\end{equation}
\end{lemma}

\begin{proof}
Following the proof of Lemma \ref{L4.8}, 	for all $\xi\in B_\HH\backslash\{0\}$, we have
\begin{equation}\label{S1-L4.9}
\begin{aligned}
&-\Delta_\HH \mu_R(\xi)+\frac{\lambda}{|\xi|_\HH^2}\psi\mu_R(\xi)\\
&=	-\sigma_\lambda(\rho)\Delta_\HH \left[\ell^\iota\left(1+\frac{\ln \rho}{\ln R}\right)\right]-2\psi  \sigma_\lambda'(\rho)\left[\ell^\iota\left(1+\frac{\ln \rho}{\ln R}\right)\right]'.
\end{aligned}
\end{equation}
Then, from  \eqref{ell-f} and \eqref{S1-L4.9},  we deduce that
\begin{equation}\label{S2-L4.9}
\begin{aligned}
&\int_{B_\HH}	\mu_R^{\frac{-1}{p-1}} \left|-\Delta_\HH \mu_R+\frac{\lambda}{|\xi|_\HH^2}\psi\mu_R\right|^{\frac{p}{p-1}}V^{\frac{-1}{p-1}}\psi^{\frac{-1}{p-1}}\,d\xi\\
&=\int_{B_\HH\left(0,\frac{1}{\sqrt R}\right)\backslash \overline{B_\HH\left(0,\frac{1}{R}\right)}}	\mu_R^{\frac{-1}{p-1}} \left|-\Delta_\HH \mu_R+\frac{\lambda}{|\xi|_\HH^2}\psi\mu_R\right|^{\frac{p}{p-1}}V^{\frac{-1}{p-1}}\psi^{\frac{-1}{p-1}}\,d\xi.
\end{aligned}	
\end{equation}
Furthermore, by \eqref{lap}, \eqref{K-f} and \eqref{S1-L4.9}, for all $\xi\in B_\HH\left(0,\frac{1}{\sqrt R}\right)\backslash \overline{B_\HH\left(0,\frac{1}{R}\right)}$, we obtain
$$
\begin{aligned}
&\left|-\Delta_\HH \mu_R(\xi)+\frac{\lambda}{|\xi|_\HH^2}\psi\mu_R(\xi)\right|\\
&\leq K(\xi) \left|\Delta_\HH \left[\ell^\iota\left(1+\frac{\ln \rho}{\ln R}\right)\right]\right|+C \psi K(\xi) \frac{|\sigma_\lambda'(\rho)|}{\sigma_\lambda(\rho)}\left|\left[\ell^\iota\left(1+\frac{\ln \rho}{\ln R}\right)\right]'\right| \\
&\leq K(\xi)\left(\left|\Delta_\HH \left[\ell^\iota\left(1+\frac{\ln \rho}{\ln R}\right)\right]\right|+C R\psi\left|\left[\ell^\iota\left(1+\frac{\ln \rho}{\ln R}\right)\right]'\right|\right)\\
&\leq C\frac{1}{\rho^2\ln R} K(\xi)\psi \ell^{\iota-2}\left(1+\frac{\ln \rho}{\ln R}\right),
\end{aligned}
$$
which implies (recall that  $0\leq \ell\leq 1$) that
\begin{equation}\label{S3-L4.9}
\begin{aligned}
&\mu_R^{\frac{-1}{p-1}} \left|-\Delta_\HH \mu_R+\frac{\lambda}{|\xi|_\HH^2}\psi\mu_R\right|^{\frac{p}{p-1}}V^{\frac{-1}{p-1}}\psi^{\frac{-1}{p-1}}\\
 &\leq C(\ln R)^{\frac{-p}{p-1}}V^{\frac{-1}{p-1}}K\psi \rho^{\frac{-2p}{p-1}}\ell^{\iota-\frac{2p}{p-1}}\left(1+\frac{\ln \rho}{\ln R}\right)\\
&\leq C(\ln R)^{\frac{-p}{p-1}}V^{\frac{-1}{p-1}}K\psi \rho^{\frac{-2p}{p-1}}.
\end{aligned}
\end{equation}
Finally, \eqref{est-L4.9} follows from \eqref{S2-L4.9} and \eqref{S3-L4.9}.
\end{proof}

We now use \eqref{J1}, \eqref{J2}, Lemmas \ref{L4.5}, \ref{L4.6} and \ref{L4.8} to obtain the following estimates.

\begin{lemma}\label{L4.10}
Let $V^{\frac{-1}{p-1}}\in L^1_{\loc}(B_\HH\backslash\{0\},K\psi\,d\xi)$. Then	
$$
J_1(\varphi)\leq C T^{1-\frac{kp}{p-1}}\int_{B_H\backslash \overline{B_\HH\left(0,\frac{1}{2R}\right)}}V^{\frac{-1}{p-1}}K\psi\,d\xi
$$
and
$$
J_2(\varphi)\leq CT R^{\frac{2p}{p-1}}\int_{B_\HH\left(0,\frac{1}{R}\right)\backslash \overline{B_\HH\left(0,\frac{1}{2R}\right)}}	V^{\frac{-1}{p-1}}K \psi \,d\xi,
$$
where $\varphi$ is defined by \eqref{testf1}.
\end{lemma}

Similarly, from \eqref{J1}, \eqref{J2}, Lemmas \ref{L4.5}, \ref{L4.7} and  \ref{L4.9}, we deduce the following estimates.

\begin{lemma}\label{L4.11}
Let $V^{\frac{-1}{p-1}}\in L^1_{\loc}(B_\HH\backslash\{0\},K\psi\,d\xi)$. Then
$$
J_1(\varphi)\leq C T^{1-\frac{kp}{p-1}}\int_{B_H\backslash \overline{B_\HH\left(0,\frac{1}{R}\right)}}V^{\frac{-1}{p-1}}K\psi\,d\xi
$$	
and
$$
J_2(\varphi)\leq CT(\ln R)^{\frac{-p}{p-1}}\int_{B_\HH\left(0,\frac{1}{\sqrt R}\right)\backslash \overline{B_\HH\left(0,\frac{1}{R}\right)}}	V^{\frac{-1}{p-1}}K \psi |\xi|_\HH^{\frac{-2p}{p-1}}\,d\xi,
$$
where $\varphi$ is defined by \eqref{testf2}.
\end{lemma}

\section{Proofs of the main results}\label{sec5}
In this section, we provide the proofs of Theorems \ref{T1} and \ref{T2}. 

\subsection{Proof of Theorem \ref{T1}}
Suppose that $u$ is a weak solution to  \eqref{P}--\eqref{BC}. For $\iota,T,R\gg 1$, let $\varphi$ be the function defined by \eqref{testf1}. By Lemma \ref{L4.2}, we know that $\varphi$ is an admissible test function. Furthermore, by Lemma \ref{L4.10}, we have $J_i(\varphi)<\infty$, $i=1,2$. Then, making use of Lemma \ref{L4.4}, we obtain
\begin{equation}\label{S1-T1}
-\int_{\Gamma} f A\nabla_{\RR^{2N+1}}\varphi\cdot n\,{{dH_{2N}}} \,dt
\leq C \sum_{i=1}^2 J_i(\varphi).
\end{equation}
On the other hand, by \eqref{testf1}, we have
$$
\begin{aligned}
-\int_{\Gamma} f A\nabla_{\RR^{2N+1}}\varphi\cdot n\,{{dH_{2N}}} \,dt&=-\left(\int_0^T \vartheta^\iota\left(\frac{t}{T}\right)\,dt\right)\left(\int_{\partial B_\HH } f A\nabla_{\RR^{2N+1}}\gamma_R\cdot n\,{{dH_{2N}}}\right)\\
&=-CT  \int_{\partial B_\HH } f A\nabla_{\RR^{2N+1}}K\cdot n\,{{dH_{2N}}},
\end{aligned}
$$
which implies by \eqref{nder} that 
\begin{equation}\label{S2-T1}
-\int_{\Gamma} f A\nabla_{\RR^{2N+1}}\varphi\cdot n\,{{dH_{2N}}} \,dt=CT 	\int_{\partial B_\HH}f\frac{\psi}{|\nabla_{\RR^{2N+1}}|\xi|_\HH|}\,{{dH_{2N}}}.
\end{equation}
Next, it follows from Lemma \ref{L4.10}, \eqref{S1-T1} and \eqref{S2-T1} that 
\begin{equation}\label{S3-T1}
\begin{aligned}
&\int_{\partial B_\HH}f\frac{\psi}{|\nabla_{\RR^{2N+1}}|\xi|_\HH|}\,{{dH_{2N}}}\\
&\leq C\left(T^{-\frac{kp}{p-1}}\int_{B_H\backslash \overline{B_\HH\left(0,\frac{1}{2R}\right)}}V^{\frac{-1}{p-1}}K\psi\,d\xi+ R^{\frac{2p}{p-1}}\int_{B_\HH\left(0,\frac{1}{R}\right)\backslash \overline{B_\HH\left(0,\frac{1}{2R}\right)}}	V^{\frac{-1}{p-1}}K \psi \,d\xi\right).
\end{aligned}
\end{equation}
Let 
$$
\eta(R):=\int_{B_H\backslash \overline{B_\HH\left(0,\frac{1}{2R}\right)}}V^{\frac{-1}{p-1}}K\psi\,d\xi,\quad R\gg 1. 
$$
Remark that $\eta$ is a nondecreasing function. Then  we have two possible cases.\\
Case 1: If 
$$
\lim_{R\to \infty}\eta(R)<\infty, 
$$
then taking $T=R$ in \eqref{S3-T1}, we obtain
$$
\int_{\partial B_\HH}f\frac{\psi}{|\nabla_{\RR^{2N+1}}|\xi|_\HH|}\,{{dH_{2N}}}
\leq C\left(R^{-\frac{kp}{p-1}}+ R^{\frac{2p}{p-1}}\int_{B_\HH\left(0,\frac{1}{R}\right)\backslash \overline{B_\HH\left(0,\frac{1}{2R}\right)}}	V^{\frac{-1}{p-1}}K \psi \,d\xi\right).
$$ 
Passing to the infimum limit as $R\to \infty$  in the above inequality, we obtain by \eqref{cd1-bl} that 
$$
\int_{\partial B_\HH}f\frac{\psi}{|\nabla_{\RR^{2N+1}}|\xi|_\HH|}\,{{dH_{2N}}}\leq 0,
$$
which contradicts the fact that $f\in L^{1,+}(\partial B_\HH)$.\\
Case 2: If 
\begin{equation}\label{eta-lim}
\lim_{R\to \infty}\eta(R)=+\infty, 
\end{equation}
then taking $T=[\eta(R)]^\theta$, where 
\begin{equation}\label{thet}
\theta>\frac{p-1}{kp},
\end{equation}
\eqref{S3-T1} reduces to  
$$
\int_{\partial B_\HH}f\frac{\psi}{|\nabla_{\RR^{2N+1}}|\xi|_\HH|}\,{
{dH_{2N}}}
\leq C\left([\eta(R)]^{1-\frac{\theta k p}{p-1}}+ R^{\frac{2p}{p-1}}\int_{B_\HH\left(0,\frac{1}{R}\right)\backslash \overline{B_\HH\left(0,\frac{1}{2R}\right)}}	V^{\frac{-1}{p-1}}K \psi \,d\xi\right).
$$
Hence, passing to the infimum limit as $R\to \infty$  in the above inequality, we reach by \eqref{cd1-bl}, \eqref{eta-lim} and \eqref{thet} a contradiction with  $f\in L^{1,+}(\partial B_\HH)$.  This completes the proof of Theorem \ref{T1}. \hfill $\square$

\subsection{Proof of Theorem \ref{T2}}
(I) We first consider the case when
\begin{equation}\label{I.case1}
\lambda \geq -\left(\frac{Q-2}{2}\right)^2\quad\mbox{and}\quad     (Q-2+\alpha^-)p>Q+a+\alpha^-.
\end{equation}
By \eqref{K-f}, one has
$$
K(\xi)\leq C |\xi|_\HH^{\alpha^-}\ln R,\quad \xi\in B_\HH\left(0,\frac{1}{R}\right)\backslash \overline{B_\HH\left(0,\frac{1}{2R}\right)},
$$
which implies by \eqref{2pt1} and \eqref{V-f}  that 
$$
\begin{aligned}
\int_{B_\HH\left(0,\frac{1}{R}\right)\backslash \overline{B_\HH\left(0,\frac{1}{2R}\right)}}	V^{\frac{-1}{p-1}}K \psi \,d\xi &\leq C \ln R	\int_{B_\HH\left(0,\frac{1}{R}\right)\backslash \overline{B_\HH\left(0,\frac{1}{2R}\right)}}|\xi|_\HH^{\alpha^--\frac{a}{p-1}}\psi\,d\xi\\
&= C \ln R \int_{\frac{1}{2R}}^{\frac{1}{R}}\rho^{\alpha^--\frac{a}{p-1}+Q-1}\,d\rho\\
&\leq C R^{\frac{a}{p-1}-Q-\alpha^-} \ln R,
\end{aligned}
$$
which yields
\begin{equation}\label{S1-T2}
R^{\frac{2p}{p-1}}\int_{B_\HH\left(0,\frac{1}{R}\right)\backslash \overline{B_\HH\left(0,\frac{1}{2R}\right)}}	V^{\frac{-1}{p-1}}K \psi \,d\xi\leq C R^{\frac{a+2p}{p-1}-Q-\alpha^-} \ln R.
\end{equation}
Observe that due to \eqref{I.case1}, one has
$$
\frac{a+2p}{p-1}-Q-\alpha^-<0.
$$
Hence, passing to the limit as $R\to \infty$ in \eqref{S1-T2}, we obtain 
$$
\lim_{R\to \infty}R^{\frac{2p}{p-1}}\int_{B_\HH\left(0,\frac{1}{R}\right)\backslash \overline{B_\HH\left(0,\frac{1}{2R}\right)}}	V^{\frac{-1}{p-1}}K \psi \,d\xi=0.
$$
Hence, from Theorem \ref{T1}, we deduce that \eqref{P}--\eqref{BC} has no weak solution. 

We next consider the critical case 
\begin{equation}\label{I.case2}
\lambda > -\left(\frac{Q-2}{2}\right)^2\quad\mbox{and}\quad  (Q-2+\alpha^-)p=Q+a+\alpha^-. 	
\end{equation}
In this case, as in the proof of Theorem \ref{T1}, we argue by contradiction by supposing that 
$u$ is a weak solution to  \eqref{P}--\eqref{BC}. For $\iota,T,R\gg 1$, let $\varphi$ be the function defined by \eqref{testf2}. By Lemma \ref{L4.3}, we know that $\varphi$ is an admissible test function. Furthermore, by Lemma \ref{L4.11}, we have $J_i(\varphi)<\infty$, $i=1,2$. Then, from Lemma \ref{L4.4}, we obtain \eqref{S1-T1}. Proceeding as in the proof of Theorem \ref{T1}, by \eqref{nder}, \eqref{testf2} and \eqref{jalh}, we obtain \eqref{S2-T1}. Next, by  \eqref{S1-T1},  \eqref{S2-T1} and Lemma \ref{L4.11}, we get 
\begin{eqnarray}\label{HHH}
C\int_{\partial B_\HH}f\frac{\psi}{|\nabla_{\RR^{2N+1}}|\xi|_\HH|}\,{{dH_{2N}}}
&\leq &  T^{-\frac{kp}{p-1}}\int_{B_H\backslash \overline{B_\HH\left(0,\frac{1}{R}\right)}}V^{\frac{-1}{p-1}}K\psi\,d\xi\\
\nonumber && + (\ln R)^{\frac{-p}{p-1}}\int_{B_\HH\left(0,\frac{1}{\sqrt R}\right)\backslash \overline{B_\HH\left(0,\frac{1}{R}\right)}}	V^{\frac{-1}{p-1}}K |\xi|_\HH^{\frac{-2p}{p-1}}\psi \,d\xi.
\end{eqnarray}
Furthermore, by \eqref{lap}, \eqref{K-f} and \eqref{I.case2}, we have
$$
\begin{aligned}
\int_{B_\HH\left(0,\frac{1}{\sqrt R}\right)\backslash \overline{B_\HH\left(0,\frac{1}{R}\right)}}	V^{\frac{-1}{p-1}}K |\xi|_\HH^{\frac{-2p}{p-1}}\psi \,d\xi &\leq  C  \int_{B_\HH\left(0,\frac{1}{\sqrt R}\right)\backslash \overline{B_\HH\left(0,\frac{1}{R}\right)}}|\xi|_\HH^{\alpha^--\frac{a+2p}{p-1}}	\psi\,d\xi\\
&= C  \int_{\frac{1}{R}}^{\frac{1}{\sqrt R}} \rho^{Q-1+\alpha^--\frac{a+2p}{p-1}}\,d\rho\\
&=C\int_{\frac{1}{R}}^{\frac{1}{\sqrt R}}\rho^{-1}\,d\rho\\
&=C\ln R,
\end{aligned}
$$
which implies by \eqref{HHH} that 
$$
\int_{\partial B_\HH}f\frac{\psi}{|\nabla_{\RR^{2N+1}}|\xi|_\HH|}\,{{dH_{2N}}}\leq C\left(T^{-\frac{kp}{p-1}}\int_{B_H\backslash \overline{B_\HH\left(0,\frac{1}{R}\right)}}V^{\frac{-1}{p-1}}K\psi\,d\xi+(\ln R)^{\frac{-1}{p-1}}\right).	
$$
Finally, proceeding as in the proof of Theorem \ref{T1}, we reach a contradiction with $f\in L^{1,+}(\partial B_\HH)$.  Consequently, if $f\in L^{1,+}(\partial B_\HH)$ and \eqref{cd-bl-V} holds, then \eqref{P}--\eqref{BC} admits no weak solution. This completes the proof of part (I) of Theorem \ref{T2}. \\[4pt]
(II) We first consider the case 
\begin{equation}\label{case1-ex}
\lambda>-\left(\frac{Q-2}{2}\right)^2.
\end{equation}
Let us consider the polynomial function
$$
P(\tau):=-\tau^2+(Q-2)\tau+\lambda,\quad \tau\in \RR.
$$
Let 
$$
\tau_1:=Q-2+\alpha^-\quad\mbox{and}\quad \tau_2:=Q-2+\alpha^+.
$$
By \eqref{case1-ex}, it can be easily seen that $\tau_1$ and $\tau_2$ are the roots of $P(\tau)$ with $\tau_1<\tau_2$. Let $\tau$ and $\varepsilon$ be two real numbers satisfying:
\begin{equation}\label{tau-ppts}
\tau_1<\tau<\min\left\{\frac{a+2}{p-1},\tau_2\right\}	
\end{equation}
and
\begin{equation}\label{eps-ppts}
0<\varepsilon<[P(\tau)]^{\frac{1}{p-1}}.	
\end{equation}
Notice that due to \eqref{cd-ex-V}, one has $\tau_1<\frac{a+2}{p-1}$, which shows that the set of $\tau$ satisfying \eqref{tau-ppts} is nonempty. Furthermore, for all $\tau$ satisfying \eqref{tau-ppts}, one has $P(\tau)>0$, which implies that the set of $\varepsilon$ satisfying \eqref{eps-ppts} is nonempty. We consider functions of the form 
\begin{equation}\label{usol1}	
u(\xi):=\varepsilon \rho^{-\tau},\quad \xi\in B_\HH\backslash\{0\},\,\, \rho=|\xi|_\HH. 
\end{equation}
Making use of \eqref{lap}, after elementary calculations we get
\begin{equation}\label{calc-jl}
-\frac{1}{\psi}\Delta_\HH u+\frac{\lambda}{|\xi|_\HH^2}u=\varepsilon P(\tau)\rho^{-\tau-2},\quad 0<\rho=|\xi|_\HH<1.
\end{equation}
Then, it follows from \eqref{tau-ppts}, \eqref{eps-ppts}, \eqref{usol1} and \eqref{calc-jl} that for all $\xi\in B_\HH\backslash\{0\}$,
$$
\begin{aligned}
-\frac{1}{\psi}\Delta_\HH u+\frac{\lambda}{|\xi|_\HH^2}u&= \left(\rho^a \varepsilon^p  \rho^{-\tau p}\right) \varepsilon^{1-p}P(\tau)\rho^{\tau (p-1)-(a+2)}\\
&=V(\xi)u^p(\xi)\varepsilon^{1-p}P(\tau)\rho^{\tau (p-1)-(a+2)}\\
&\geq V(\xi)u^p(\xi). 
\end{aligned}
$$
Furthermore, one has 
$$
u(\xi)=\varepsilon,\quad \xi\in \partial B_\HH. 
$$
Consequently, for all $\tau$ satisfying \eqref{tau-ppts} and $\varepsilon$ satisfying \eqref{eps-ppts}, the function $u$ defined in $B_\HH\backslash\{0\}$ by \eqref{usol1} is a solution to \eqref{P}--\eqref{BC} with $f=\varepsilon$. 

We next consider the case
\begin{equation}\label{case2-ex}
\lambda=-\left(\frac{Q-2}{2}\right)^2.
\end{equation}
For $\beta\in (0,1)$, let us introduce the function 
$$
h_\beta(s):=s^{\frac{(p-1)(Q-2)}{2}-(a+2)}\left(1-\ln s\right)^{-\beta(p-1)-2},\quad 0<s\leq 1. 
$$
Due to \eqref{cd-ex-V} and \eqref{case2-ex}, we have 
$$
\frac{(p-1)(Q-2)}{2}-(a+2)<0, 
$$
which implies that 
$$
\lim_{s\to 0^+}h_\beta(s)=+\infty. 
$$
Consequently, there exists $s_0\in (0,1]$ such that 
$$
h_\beta(s_0)=\min_{0<s\leq 1}h_\beta(s)>0.
$$

For all $\beta\in (0,1)$ and 
\begin{equation}\label{eps2}
0<\varepsilon< \left[\beta(1-\beta)h_\beta(s_0)\right]^{\frac{1}{p-1}}, 
\end{equation}
we consider functions of the form
\begin{equation}\label{usol2}
u(\xi):=\varepsilon \rho^{\frac{2-Q}{2}}\left(1-\ln \rho\right)^\beta,\quad \xi\in B_\HH\backslash\{0\},\,\, \rho=|\xi|_\HH. 	
\end{equation}
Elementary calculations show that 
\begin{equation}\label{calc-jl2}
-\frac{1}{\psi}\Delta_\HH u+\frac{\lambda}{|\xi|_\HH^2}u=\varepsilon \beta(1-\beta)\rho^{-\frac{Q}{2}-1}\left(1-\ln \rho\right)^{\beta-2},\quad 0<\rho=|\xi|_\HH<1.	
\end{equation}
Then, using  \eqref{eps2}, \eqref{usol2} and \eqref{calc-jl2}, we obtain that for all $\xi\in B_\HH\backslash\{0\}$,
$$
\begin{aligned}
-\frac{1}{\psi}\Delta_\HH u+\frac{\lambda}{|\xi|_\HH^2}u& =\left(V(\xi) \varepsilon^p \rho^{\frac{(2-Q)p}{2}}\left(1-\ln \rho\right)^{\beta p}\right) \varepsilon^{1-p}\beta(1-\beta)\rho^{\frac{(p-1)(Q-2)}{2}-(a+2)}\left(1-\ln \rho\right)^{-\beta(p-1)-2}\\
&= V(\xi) u^p(\xi) \varepsilon^{1-p}\beta(1-\beta)h_\beta(\rho)\\
&\geq V(\xi) u^p(\xi) \varepsilon^{1-p}\beta(1-\beta)h_\beta(s_0)\\
&\geq  V(\xi) u^p(\xi).
\end{aligned}
$$ 
Furthermore, one has 
$$
u(\xi)=\varepsilon,\quad \xi\in \partial B_\HH. 
$$
Consequently, for all $\beta\in (0,1)$ and $\varepsilon$  satisfying  \eqref{eps2}, the function $u$ defined in $B_\HH\backslash\{0\}$ by \eqref{usol2} is a solution to \eqref{P}--\eqref{BC} with $f=\varepsilon$. The proof of part (II) of Theorem \ref{T2} is then completed.  \hfill $\square$.

\section*{Declaration of competing interest}
	The authors declare that there is no conflict of interest.

\section*{Data Availability Statements} The manuscript has no associated data.

\section*{Acknowledgments}
MJ is supported by Researchers Supporting Project number (RSP2024R57), King Saud University, Riyadh, Saudi Arabia.
MR and BT were supported by the FWO Odysseus 1 grant G.0H94.18N: Analysis and Partial Differential Equations and by the Methusalem programme of the Ghent University Special Research Fund (BOF) (Grant number 01M01021). MR is also supported by EPSRC grant EP/R003025/2 and EP/V005529/1. BS is supported by Researchers Supporting Project number (RSP2023R4), King Saud University, Riyadh, Saudi Arabia. BT is also supported by the Science Committee of the Ministry of Education and Science of the Republic of Kazakhstan (Grant No. AP14869090).

\end{document}